\documentclass[reqno,11pt]{amsart}
\usepackage{amsmath,amssymb,latexsym,soul,cite,mathrsfs}
\usepackage{color,enumitem,graphicx}
\usepackage[colorlinks=true,urlcolor=blue,
citecolor=red,linkcolor=blue,linktocpage,pdfpagelabels,
bookmarksnumbered,bookmarksopen]{hyperref}
\usepackage[english]{babel}
\usepackage[left=2.6cm,right=2.6cm,top=2.8cm,bottom=2.8cm]{geometry}
\usepackage[hyperpageref]{backref}
\usepackage{lipsum}

 \usepackage{mathtools, nccmath}

\usepackage{bbm}
\usepackage{tcolorbox}

\newtheorem{theorem}{Theorem}
\newtheorem{lemma}[theorem]{Lemma}
\newtheorem{corollary}[theorem]{Corollary}
\newtheorem{proposition}[theorem]{Proposition}
\newtheorem{remark}[theorem]{Remark}

\newcommand{\innerthmname}{}

\theoremstyle{definition}

\makeatletter
\def\namedlabel#1#2{\begingroup
	#2%
	\def\@currentlabel{#2}%
	\phantomsection\label{#1}\endgroup
}
\makeatother

\usepackage{etoolbox}
\makeatletter
\patchcmd{\@maketitle}
  {\ifx\@empty\@dedicatory}
  {\ifx\@empty\@date \else {\vskip3ex \centering\footnotesize\@date\par\vskip1ex}\fi
   \ifx\@empty\@dedicatory}
  {}{}
\patchcmd{\@adminfootnotes}
  {\ifx\@empty\@date\else \@footnotetext{\@setdate}\fi}
  {}{}{}
\makeatother

\newcommand{\dive}{\mathrm{div}}

\newcommand{\ep}{\varepsilon}

\newcommand{\R}{\mathbb{R}}

\newcommand{\ns}{\hspace{-5pt}}

\DeclareMathOperator{\Ric}{Ric}

\newcommand{\ov}[1]{\widetilde{#1}}

\allowdisplaybreaks[2]

\numberwithin{equation}{section}
\numberwithin{theorem}{section}

\title[A new Duhamel-type principle with applications to geometric (in)equalities]{A new Duhamel-type principle with applications to geometric (in)equalities}

\author[M. Caselli]{Michele Caselli}
\author[L. Gennaioli]{Luca Gennaioli}

\address[M. Caselli]{Princeton University --  Fine Hall, 304 Washington Rd, Princeton, NJ 08540, USA
	\newline\indent 
	University of Sydney -- Quadrangle A14, Camperdown, NSW 2006, Australia}
\email{\href{mailto:mc3147@princeton.edu}{mc3147@princeton.edu}}

\address[L. Gennaioli]{University of Warwick -- Zeeman Building, Coventry, CV4 7AL, UK}
\email{\href{mailto:luca.gennaioli@warwick.ac.uk}{luca.gennaioli@warwick.ac.uk}}

\makeatletter
\def\l@subsection{\@tocline{2}{0pt}{2.5pc}{5pc}{}}
\def\l@subsubsection{\@tocline{2}{0pt}{5pc}{7.5pc}{}}
\makeatother

\date{\today}

\begin{document}

\begin{abstract}
    We introduce a simple new method, based on the Caffarelli-Silvestre extension and a Duhamel-type formula, to derive exact pointwise identities for fractional commutators and nonlinear compositions associated with the fractional Laplacian on general Riemannian manifolds.

As applications, we obtain a pointwise fractional Leibniz rule, a fractional Bochner's formula with an explicit Ricci curvature term, apparently the first of this kind, and exact remainders in the Córdoba-Córdoba and Kato inequalities for the fractional Laplacian. All these formulas are new even in the Euclidean space. 

\end{abstract}

	\maketitle

 \tableofcontents

 \section{Introduction} 

 In this work, we present a simple method to derive several exact pointwise identities for commutators and compositions of the fractional Laplacian. The key idea is that, representing the fractional Laplacian as the weighted normal derivative of the Caffarelli-Silvestre (CS) extension, the error terms measuring the failure of classical differential identities can be encoded as boundary values of solutions to a suitable inhomogeneous problem in the extension. Once this is observed, a Duhamel-type formula gives the desired identities in a unified and transparent way. 

 As an application, we obtain for the first time a fractional Bochner's formula for the fractional Laplacian, that is, an exact pointwise identity with an explicit Ricci curvature term on general (possibly non-compact) manifolds, without any assumption on the sign of the curvature of the underlying manifold. This is somewhat surprising, as the involved operators are non-local. On the other hand, for the formulae to even make sense, we need to ask for some integrability of the Ricci tensor when applied to horizontal gradients of the CS extension.
 
Also, remarkably, we identify an explicit remainder in the Kato inequality for the fractional Laplacian. This remainder is not merely nonnegative, but it is an explicit combination of integrals over the zero level sets of the function and of its extension.

\subsection{Relation to prior work}

The extension point of view for the fractional Laplacian, which
may have originated in the probability literature in \cite{ProbExtension}, classically goes back to the seminal work of Caffarelli and Silvestre \cite{CafSiextension}, and in the general self-adjoint setting to Stinga and Torrea \cite{StingaTorrea}.

On the one hand, extension methods have been used effectively in the study of fractional commutator estimates, notably in \cite{LS-sharpcomm} (see also \cite{Shen-comm-extension}), and to prove pointwise inequalities of Córdoba-Córdoba and Kato-type for fractional operators \cite{CafSire}. On the manifolds side, the fractional Laplacian has been studied through several complementary viewpoints, including spectral and heat-semigroup definitions, heat-kernel/singular-integral representation formulas \cite{AlonsoCordobaMartinez}, and via the extension in both compact \cite{CFSfrac} and noncompact settings \cite{flapHyperbolic}.

On the other hand, even in the Euclidean space, the extension-Duhamel viewpoint has not been used to obtain exact pointwise identities in a unified way. Indeed, the contribution of the present work differs from previous literature: rather than deriving estimates from the extension problem, we isolate a simple mechanism that yields exact pointwise identities for several fractional defects. 

 To the best of our knowledge, our Theorem \ref{thm: Bochner ricci} is the first result relating a Bochner-type commutator for the fractional Laplacian to an explicit Ricci curvature term, in the spirit of the classical Bochner's formula. Nevertheless, even on $\R^n$, we point out that the relative curvature dimension inequality ${\rm CD}(K, N)$ for the fractional Laplacian fails for every finite $N$; see \cite{Sperner-Weber-Rico}. 

 Similarly, Kato-type inequalities for the fractional Laplacian are well known, but to our knowledge, these results are formulated as inequalities and do not provide an explicit Kato formula that identifies the nonnegative defect in terms of the nodal sets of the function and its extension. In this regard, our Theorem \ref{thm: Kato} seems to be the first result that is able to do so.

 \subsection{Main results} 

  Let $(M,g)$ be a complete Riemannian manifold without boundary, and let $\ov M:= M \times ( 0,\infty)$ be the extended space endowed with the product metric. We refer to Section \ref{sec 2} for a list of all the notations and a precise definition of the spaces involved in our results.

 As outlined above, the starting point of the present work is a simple yet very useful observation that will serve as the common underlying mechanism for all the results in the paper. If we compare the CS extension of a composite quantity (a product, a squared gradient, or a nonlinear composition) with the corresponding composite of extensions, then the difference solves an inhomogeneous problem in the extended space $\ov M$ with zero Dirichlet data on $M\times \{0\}$. The forcing term is exactly the interaction term arising in the nonlocal procedure (a Leibniz rule defect, a Bochner/Ricci curvature term, or a nonlinear remainder). By a Duhamel-type principle, such a problem admits an explicit solution in terms of the fractional Poisson semigroup, and taking the weighted normal derivative at zero yields the desired identity on the base space $M$.

\begin{proposition}[Duhamel-type formula for the extension]\label{prop: rep formula solutions}
    Let $s\in (0,1)$, and $F \in L^1(\ov M)$, $V \in \widehat{H}^1_w (\widetilde M )$ be smooth functions, with $V$ solution of
    \begin{equation*}
    \begin{cases} 
      \widetilde \dive (z^{1-2s} \widetilde \nabla V )  = F     &   \mbox{in } \widetilde{M} ,  \\ V(\cdot, 0)=0 &  \mbox{on } M . 
    \end{cases}
    \end{equation*}
    Then
    \begin{equation*}
        \lim_{z \uparrow \infty }  z^{1-2s} \partial_z V (\cdot, z) - \lim_{z \downarrow 0 } z^{1-2s} \partial_z V (\cdot, z) = \int_0^\infty \mathcal{P}^{(s)}_z \big( F(\cdot,z) \big)  \, dz \qquad \mbox{in } \mathscr{D}'(M) .  
    \end{equation*}
     Here $\mathcal{P}^{(s)}_z$ denotes the fractional Poisson operator \eqref{eq: frac Poisson def}, which acts on $F(\cdot, z)$ for fixed $z>0$. 
\end{proposition}

In all the applications of Proposition \ref{prop: rep formula solutions} in this work (see Section \ref{sec: appl} below), it will be applied in cases when the boundary term at infinity vanishes, in which case the formula reads as 
    \begin{equation}\label{eq: Duhamel infinity zero}
        - \lim_{z \downarrow 0 } z^{1-2s} \partial_z V (\cdot, z) = \int_0^\infty \mathcal{P}^{(s)}_z \big( F(\cdot,z) \big)  \, dz .  
\end{equation}
 Indeed, the proofs of Theorems \ref{thm: baby Boch Gamma_1}, \ref{thm: Bochner ricci}, and \ref{thm: C-C general phi} follow an identical scheme: identify the interaction as a forcing term in the problem in the extended space, and go back to the base via \eqref{eq: Duhamel infinity zero}.

\subsection{Applications}\label{sec: appl} We now present several direct applications of Proposition \ref{prop: rep formula solutions}, marked below with the symbol $\bullet$. Let $-\Delta$ be the Laplace-Beltrami operator of $M$. For $s\in(0,1]$ let
\begin{equation*}
    \Lambda^{\hspace{-1pt} s} := -(-\Delta)^s , 
\end{equation*}
where $(-\Delta)^s$ denotes the fractional Laplacian in the sense of spectral theory (see Section \ref{sec 2} and Remark \ref{rem: which manifolds}). Thus, with our notation $\Lambda^{\hspace{-1pt} s}$ is a nonpositive operator. 

  \medskip

\noindent\textbf{$\bullet$ Pointwise commutator identity for the fractional Laplacian.}  Define the first-order carré du champ (i.e., the Leibniz defect) on smooth functions by
\begin{equation*}
\Gamma_1(u,v):=\frac12\Big( \Lambda^{\hspace{-1pt} s}(uv)-u\Lambda^{\hspace{-1pt} s} v-v\Lambda^{\hspace{-1pt} s} u\Big) , \qquad \Gamma_1(u):=\Gamma_1(u,u) .
\end{equation*}
For general $u,v \in H^s(M)$, this has to be understood in the sense of distributions as
\begin{equation*}
    \Gamma_1(u,v)(\varphi) := \frac12\Big( \langle u v,  \Lambda^{\hspace{-1pt} s} \varphi \rangle_{L^2(M)}- \langle \Lambda^{\hspace{-1pt} s/2}  (u \varphi),  \Lambda^{\hspace{-1pt} s/2} v \rangle_{L^2(M)} -\langle \Lambda^{\hspace{-1pt} s/2} (v \varphi), \Lambda^{\hspace{-1pt} s/2} u \rangle_{L^2(M)} \Big) , 
\end{equation*}
for all $\varphi \in C_c^\infty(M)$. 

  \begin{theorem}[Pointwise fractional Leibniz rule]\label{thm: baby Boch Gamma_1} Let $s\in (0,1)$. Let $u,v \in H^{s}(M)$ and $U,V \in {\widehat H}^1_w (\widetilde M)$ denote their CS extensions. Then    
\begin{equation}\label{eq: Gamma_1 rep}
    \Gamma_1(u,v) =  \beta_s \int_0^\infty \mathcal{P}_z^{(s)}( \widetilde \nabla U (\cdot, z) \cdot \widetilde \nabla V (\cdot, z) ) z^{1-2s} \, dz    \quad \mbox{ in } \mathscr{D}'(M) , 
\end{equation}
where $\beta_s > 0 $ is given by \eqref{eq: beta def}. Moreover, if $u$ and $v$ are smooth, this identity holds pointwise. 
 \end{theorem}

 We emphasize that the surprising and useful part of \eqref{eq: Gamma_1 rep} is not that $\Gamma_1(u,v)(x)$ has a pointwise representation per se. For example, for $M=\R^n$ it is straightforward (see \cite[Sec. 20]{GarofaloFracTh}) that 
\begin{equation*}
    \Gamma_1(u,v)(x) = C_{n,s} \int_{\R^n} \frac{(u(x)-u(y))(v(x)-v(y))}{|x-y|^{n+2s}} dy .  
\end{equation*}
The surprising part of \eqref{eq: Gamma_1 rep} is that $\Gamma_1(u,v)(x)$ admits a representation as an average of \emph{classical derivatives} of the extensions \emph{evaluated at the same point $x$}. This will have important consequences, for example, in the next application regarding Bochner's formula for the fractional Laplacian.

 \medskip

\noindent\textbf{$\bullet$ Fractional Bochner-type identities.} Define the second-order carré du champ 
\[
\Gamma_2(u,v):=\frac12\Big( \Lambda^{\hspace{-1pt} s} \Gamma_1(u,v)-\Gamma_1(u,\Lambda^{\hspace{-1pt} s} v)-\Gamma_1(v,\Lambda^{\hspace{-1pt} s} u)\Big),
\qquad  \Gamma_2(u):=\Gamma_2(u,u),  
\]
and the Bochner-type defects
 \begin{align*}
  \mathcal{A} (u,v) & := \frac{1}{2} \Big( \Delta \Gamma_1(u,v) - \Gamma_1(u, \Delta v ) - \Gamma_1(v , \Delta u ) \Big) , & \mathcal{A} (u)  :=  \mathcal{A} (u,u) , \\
         \mathcal{B} (u,v) & := \frac{1}{2} \Big( \Lambda^{\hspace{-1pt} s} (\nabla u \cdot \nabla v) - \nabla u \cdot \nabla \Lambda^{\hspace{-1pt} s} v - \nabla v \cdot \nabla \Lambda^{\hspace{-1pt} s} u \Big) ,  & \mathcal{B} (u)  :=  \mathcal{B} (u,u) .
     \end{align*}
In the local case $s=1$, by Bochner's formula the last two expressions both equal  
\begin{equation*}
   \mathcal{A}(u) = \mathcal{B} (u)= |\nabla^2 u|^2 + \Ric(\nabla u, \nabla u). 
\end{equation*}

One application of the techniques in this work is a clean fractional version of Bochner's formula, which explicitly accounts for the classical Ricci curvature.

 \begin{theorem}[Fractional Bochner's formula]\label{thm: Bochner ricci}
    Let $s\in (0,1)$, $u\in C_c^\infty(M)$ and $U \in {\widehat H}^1_w (\widetilde M) $ be its CS extension. Assume in addition\footnote{Observe that this assumption is not playing the role of a curvature lower bound, and instead is a necessary condition for the displayed r.h.s. to be well-defined.} that $ z^{1-2s}\big( |\nabla^2 U|^2+ {\rm Ric}(\nabla U ,\nabla U )\big) \in L^1(\widetilde{M})$. Then
     \begin{equation}\label{eq: eq for A}
          \mathcal{A}(u) =\mathcal{B}(u) = \beta_s \int_0^\infty \mathcal{P}_z^{(s)} \Big( |\ov \nabla \nabla  U(\cdot, z) |^2 + \Ric(\nabla U (\cdot, z) , \nabla U (\cdot, z) ) \Big) z^{1-2s}  dz  . 
     \end{equation}
 \end{theorem}

Observe that in the proof of \eqref{eq: eq for A}, it is not trivial to infer that the defects $\mathcal{A}(u)$ and $\mathcal{B}(u)$ are equal. Indeed, first we shall prove the equality between $\mathcal{A}(u)$ and the r.h.s. via a direct application of \eqref{eq: Gamma_1 rep} and (the classical) Bochner's formula. Secondly, we will prove the more involved equality between $\mathcal{B}(u)$ and the r.h.s., ultimately leading to $\mathcal{A}(u)=\mathcal{B}(u)$. 

As a corollary of this formula, we also obtain an explicit representation of the second-order carré du champ. 

  \begin{corollary}\label{thm: Gamma_2}
  Under the assumptions of Theorem \ref{thm: Bochner ricci} we have
     \begin{align*}
            \Gamma_2(u) & = \beta_s \int_0^\infty \mathcal{P}_z^{(s)} \Big(\mathcal{B}(U (\cdot, z)) + \Gamma_1(\partial_z U(\cdot, z)) \Big)  z^{1-2s}  dz . 
     \end{align*}
 \end{corollary}

\medskip

\noindent\textbf{$\bullet$ Exact remainder in the Córdoba-Córdoba inequality.} For $s\in (0,1)$ and $\phi \in C^2(\R)$ convex there holds the well-known Córdoba-Córdoba inequality
\begin{equation}\label{eq: C-C}
    \Lambda^{\hspace{-1pt} s} (\phi(u)) \ge \phi'(u) \Lambda^{\hspace{-1pt} s} u .
 \end{equation}
The inequality was originally proved on $\R^n$ in \cite{CordobaCordoba}, and it has since been extended in several directions, including closed manifolds \cite{CordobaMartinez}, the Dirichlet fractional Laplacian in bounded domains \cite{ConstantinIgna}, and a broad unifying approach (similar to the technique in this work but without characterizing the exact remainders) covering many settings \cite{CafSire}.

The methods in this work also yield an exact pointwise remainder in the Córdoba-Córdoba inequality \emph{for arbitrary nonlinearities}, that is not necessarily convex.

  \begin{theorem}[Exact remainder in the Córdoba-Córdoba inequality]\label{thm: C-C general phi}  Let $s\in (0,1)$, $u\in C_c^\infty(M)$ and $U \in {\widehat H}^1_w (\widetilde M)$ be its CS extension. Let $\phi \in C^2(\R)$, with the additional assumption $\phi(0)=0$ when $M$ is non-compact. Then 
\begin{equation*}
  \Lambda^{\hspace{-1pt} s}  (\phi(u)) - \phi'(u) \Lambda^{\hspace{-1pt} s} u =   \beta_s  \int_0^\infty \mathcal{P}_z^{(s)} \Big(\phi''(U)|\widetilde \nabla U|^2(\cdot, z) \Big) z^{1-2s} \, dz .
 \end{equation*}
 In particular, if $\phi$ is convex, we recover the Córdoba-Córdoba inequality \eqref{eq: C-C}. 
\end{theorem}
A direct application of
Theorem \ref{thm: C-C general phi} is a pointwise version of the classical Stroock-Varopoulos inequality on general manifolds, which reads as follows.
\begin{proposition}
\label{prop:SVI}
Let $(M,g)$ be a complete and stochastically complete Riemannian manifold, let
\(s\in(0,1)\), \(q >1 \), and \(u\in C^\infty_c(M)\). Then
\begin{equation*}
\int_{M} (|u|^{q-2}u) (-\Delta)^s u \,d\mu
\ge
\frac{4(q-1)}{q^2}
\int_{M} \big|(-\Delta)^{s/2}(|u|^{q/2})\big|^2\,d\mu .
\end{equation*}  
\end{proposition}
Indeed, it follows by applying Theorem \ref{thm: C-C general phi} to $\phi_\ep(t)=\tfrac{1}{q}(|t|^2+\ep^2)^{q/2}-\tfrac{\ep^q}{q}$ and a standard approximation argument as $\ep \downarrow 0$, and by the identity 
\begin{equation*}
   \frac{q^2}{4}|U|^{q-2} |\widetilde \nabla U|^2 =  |\widetilde\nabla(|U|^{q/2})|^2 .
\end{equation*}

\begin{corollary}[Pointwise Stroock-Varopoulos identity] Let \(s\in(0,1)\),
\(q > 1\), $u\in C_c^\infty(M)$ and \(U\in \dot H^1_w(\widetilde M)\) be its CS
extension. Then
\begin{equation}\label{eq: pointwise SV}
\frac1q \Lambda^s(|u|^q)-|u|^{q-2}u \Lambda^s u
=
\frac{4(q-1)}{q^2} \beta_s
\int_0^\infty
\mathcal{P}_z^{(s)} \left(|\widetilde\nabla(|U|^{q/2})|^2(\cdot,z)\right) z^{1-2s}\,dz . 
\end{equation}
\end{corollary}

\medskip

\noindent\textbf{$\bullet$ Exact remainder in the Kato inequality.} By virtue of Theorem \ref{thm: C-C general phi}, we shall also deduce an explicit remainder in the Kato inequality for the fractional Laplacian, originally due to Chen and Véron \cite{KatoFrac}, with later generalizations \cite{AbatangeloKato}. It is remarkable that we are able to obtain a pointwise nonnegative (in the sense of distribution) remainder in the Kato inequality for a nonlocal operator, which resembles that of the classical case.

 \begin{theorem}[Exact remainder in the Kato inequality]\label{thm: Kato}  Let $s\in (0,1)$, $u\in C_c^\infty(M)$, and $U \in {\widehat H}^1_w (\widetilde M)$ be the CS extension of $u$. Then, for all $\varphi\in{\rm Lip}_c(M)$, we have  
\begin{equation}\label{eq: remainder Kato}
  \int_M \Big(\Lambda^{\hspace{-1pt} s}  |u| - {\rm sgn}(u) \Lambda^{\hspace{-1pt} s} u\Big)\varphi \,  d\mu = 2 \beta_s\int_{\{U=0\}}|\ov \nabla U|\Phi z^{1-2s}  d\mathcal{H}^n + \int_{\{u=0\}} |\Lambda^{\hspace{-1pt} s} u| \varphi \, d\mu , 
 \end{equation}
 where ${\rm sgn}(u)$ is the sign of $u$ (with value $0$ at $u=0$) and $\Phi=\mathcal{P}_z^{(s)}\varphi$. 
\end{theorem}

The identity \eqref{eq: remainder Kato} should be thought of as the fractional analogue of the classical Kato formula for the Laplacian, in which the defect between $\Delta |u| - {\rm sgn}(u)\Delta u$ is concentrated on the zero set $\{u=0\}$. In the fractional setting, this defect decomposes into two nonnegative contributions: one supported on the zero set of the extension $U$, and one supported on the zero set of $u$. 

Thus, the failure of commutation between $\Lambda^{\hspace{-1pt} s}$ and the absolute value is encoded not only by the zero set of $u$, but also by the geometry of the zero level set of its CS extension. Lastly, note that formally the second term in \eqref{eq: remainder Kato} vanishes in the classical limit $s \uparrow 1$, since $\Delta u=0$ almost everywhere on $\{u=0\}$.

 \section{Extension framework and the Duhamel-type formula}\label{sec 2}

We fix the following notations that will be used throughout the paper. 

\begin{center}
\begin{tabular}{@{}p{0.10\linewidth}p{0.72\linewidth}@{}}
\multicolumn{2}{@{}l@{}}{}\\
$(M,g)$ & a smooth, complete Riemannian manifold without boundary. \\
$d\mu$ & Riemannian volume measure of $M$. \\
$\nabla$ & Riemannian gradient on $M$. \\
$\Delta$ & the Laplace-Beltrami operator of $M$.  \\ 
$\widetilde M $ & extended manifold $\widetilde M:=M\times (0,\infty)$ with the product metric.  \\ 
$\ov \nabla $ & gradient $\ov \nabla := (\nabla, \partial_z)$ on $\ov M$.  \\ 
$\ov \dive $ & divergence operator on $\ov M$.  \\ 
$\mathcal{L}_s$ & degenerate elliptic operator $\mathcal{L}_s U := \ov \dive (z^{1-2s} \ov \nabla U)$ on $\ov M$. 
\end{tabular}
\end{center}

\vspace{12pt}

  We refer to \cite{GrygBook} and \cite[Section 2.6]{Spec1} and the references therein for an introduction to the spectral theory of the fractional Laplacian on general spaces. Let $\sigma(-\Delta) \subset [0,\infty)$ be the spectrum of $-\Delta$ and ${E_\lambda}$ be its spectral resolvent. For $s\in (0,1)$, we have
\begin{equation*}
    {\rm Dom}((-\Delta)^{s}) = \Big\{u \in L^2(M) \, :  \int_{\sigma(-\Delta)} \lambda^{2s} \, d \| E_\lambda u \|^2 < \infty \Big\} , 
\end{equation*}
and 
\begin{align*}
    H^s(M) := \Big\{u \in L^2(M) \, :  \int_{\sigma(-\Delta)} (1+\lambda^{s}) \, d\| E_\lambda u \|^2 < \infty \Big\} . 
\end{align*}

 The Caffarelli-Silvestre extension admits a very general operator-theoretic formulation originally due to Stinga and Torrea \cite{StingaTorrea}. Indeed, even though in the first few pages in \cite{StingaTorrea} the result is presented for (second-order, nonnegative, self-adjoint) differential operators densely defined on a bounded domain of $\R^n$, from \cite[Section 2]{StingaTorrea} onwards the proof only uses the spectral theorem, so the result applies to general non-negative normal operators. This result has now been extended even to the case of ambient Banach spaces \cite{CSbanach1, CSbanach2}. 
 
  In our case, since $M$ is a complete Riemannian manifold, then the Laplacian $-\Delta$ is essentially self-adjoint on $C_c^\infty(M)$ and thus admits a canonical nonnegative self-adjoint realization. Applying the result by Stinga and Torrea in this case, we obtain that the Caffarelli-Silvestre extension realizes the spectral fractional Laplacian $(-\Delta)^{s}$ as the corresponding Dirichlet-to-Neumann operator, which we describe next.

Define the weighted homogeneous $H^1$-space
\begin{equation*}
        {\widehat H}^1_w (\widetilde M) := \Big\{ U \in L^2_{\rm loc}(\ov M) \, : \, \medint\int_{\ov M} |\ov \nabla U|^2 z^{1-2s} d\mu dz  <+\infty  \Big\}.
\end{equation*} 
Given $u\in H^s(M)$, we denote by $U$ its Caffarelli-Silvestre (CS) extension to $\widetilde M$ given by 
\begin{equation*}
    U(x,z) = \mathcal{P}_z^{(s)} u (x) :=  \int_{M} \mathcal{P}_z^{(s)}(x,y) u(y) \, d\mu(y),  
\end{equation*}
where the fractional Poisson kernel $ \mathcal{P}_z^{(s)} : M \times M \to (0, \infty)$ of $M$ can be expressed as
\begin{equation}\label{eq: frac Poisson def}
    \mathcal{P}_z^{(s)}(x,y)= \frac{z^{2s}}{4^s\Gamma(s)} \int_{0}^{\infty} H_M(x,y,t)  e^{-\frac{z^2}{4t}}\frac{dt}{t^{1+s}} , 
\end{equation} 
and $H_M$ is the heat kernel of $M$ (see \cite[Section 9.1]{GrygBook}). 

By \cite{StingaTorrea}, see also \cite{CFSfrac} for the case of closed manifolds, this extension $U$ is the unique function $U\in {\widehat H}^1_w (\widetilde M)$ solution of
\begin{equation*}
    \begin{cases}  \widetilde{ {\rm div}}(z^{1-2s} \widetilde \nabla U) = 0  &   \mbox{in } \widetilde{M} ,  \\ U(\cdot,0) = u &  \mbox{on } M . \end{cases}
\end{equation*}

Moreover, the extension realizes the spectral fractional Laplacian as 
\[
 \Lambda^{\hspace{-1pt} s} u  =  \beta_s \lim_{z\downarrow 0} z^{1-2s} \partial_z U(\cdot,z)  \qquad \mbox{in } \mathscr{D}'(M) , 
\]
where 
\begin{equation}\label{eq: beta def}
    \beta_s := \frac{2^{2s-1} \Gamma(s)}{\Gamma(1-s)} > 0 . 
\end{equation}

\begin{remark}\label{rem: which manifolds}
    In the discussion above, we have focused exclusively on the spectral fractional Laplacian and its extension formulation by Stinga and Torrea. Hence, every result in this work holds on a general complete Riemannian manifold $M$ for the spectral fractional Laplacian. A separate issue is whether this operator agrees with the fractional Laplacian defined through the principal value of a singular integral, as in the classical setting of $\R^n$. This identification holds in general in many situations of interest, for example, when 
    \begin{itemize}
        \item $M=\R^n$ by \cite{Hitguide}. 
        \item $M=\mathbb{H}^n$ (the hyperbolic space) by \cite{flapHyperbolic}.
        \item $M$ is a closed Riemannian manifold without boundary by \cite{CFSfrac}. 
    \end{itemize} 
    Moreover, the two notions agree on smooth, compactly supported functions when $M$ is stochastically complete \cite{CG24}.
\end{remark}

\subsection{Proof of the Duhamel-type formula}

\begin{proof}[Proof of Proposition \ref{prop: rep formula solutions}] Let $\varphi \in C_c^\infty(M)$ and $\Phi \in {\widehat H}^1_w (\widetilde M) $ be its Caffarelli–Silvestre extension. Fix $p \in M$ and let $\eta_k \in C_c^\infty(M)$ be a standard cutoff with $\eta_k =1$ in $B_k(p)$ and $\eta_k =0$ in $ M \setminus B_{2k}(p)$. Let also $0<\ep<R$. Multiplying the equation $  \mathcal{L}_s V = F$ by $\Phi \eta_k $ and integrating over $M\times (\ep, R)$ gives 
\begin{equation}\label{eq: integrated}
    \int_{M\times (\ep, R)} (\mathcal{L}_s V ) (\Phi \eta_k) \, d\mu dz = \int_{M\times (\ep, R)} F \Phi \eta_k \, d\mu dz . 
\end{equation}
Now we want to integrate by parts the l.h.s. twice. We have 
{ \allowdisplaybreaks
\begin{align*}
   & \int_{M\times (\ep, R)}  (\mathcal{L}_{s}V) (\Phi \eta_k) \, d\mu dz \\ & = - \int_{M\times (\ep, R)} \ov \nabla V \cdot \ov \nabla(\Phi \eta_k) z^{1-2s} d\mu dz + \int_{M} \Big[ z^{1-2s}  \Phi \partial_z V \eta_k   \Big]_{z=\ep}^{z=R}  d\mu \\ & = - \int_{M\times (\ep, R)} \ov \nabla V \cdot \ov \nabla \Phi \eta_k z^{1-2s} d\mu dz - \int_{M\times (\ep, R)} \nabla V \cdot   \nabla \eta_k \Phi z^{1-2s} d\mu dz + \int_{M} \Big[ z^{1-2s}  \Phi \partial_z V \eta_k   \Big]_{z=\ep}^{z=R}  d\mu \\ & =  \int_{M\times (\ep, R)}  V \widetilde \dive ( \ov \nabla \Phi \eta_k z^{1-2s} ) d\mu dz - \int_{M} \Big[ z^{1-2s}  V \partial_z \Phi \eta_k  \Big]_{z=\ep}^{z=R}  d\mu - \int_{M\times (\ep, R)} \nabla V \cdot   \nabla \eta_k \Phi z^{1-2s} d\mu dz \\ & \hspace{12pt} + \int_{M} \Big[ z^{1-2s}  \Phi \partial_z V \eta_k   \Big]_{z=\ep}^{z=R}  d\mu \\ &= \int_{M\times (\ep, R)} V \nabla \Phi \cdot \nabla \eta_k z^{1-2s} d\mu dz - \int_{M\times (\ep, R)} \nabla V \cdot   \nabla \eta_k \Phi z^{1-2s} d\mu dz - \int_{M} \Big[ z^{1-2s}  V \partial_z \Phi \eta_k  \Big]_{z=\ep}^{z=R}  d\mu \\ &  \hspace{12pt} + \int_{M} \Big[ z^{1-2s}  \Phi \partial_z V \eta_k   \Big]_{z=\ep}^{z=R}  d\mu \\ &= \int_{M\times (\ep, R)} V \nabla \Phi \cdot \nabla \eta_k z^{1-2s} d\mu dz - \int_{M\times (\ep, R)} \nabla V \cdot   \nabla \eta_k \Phi z^{1-2s} d\mu dz \\ &  \hspace{12pt} + R^{1-2s} \int_M (\Phi \partial_z V - V \partial_z \Phi)(\cdot, R)\eta_k \, d\mu - \ep^{1-2s} \int_M (\Phi \partial_z V - V \partial_z \Phi)(\cdot, \ep)\eta_k \, d\mu 
   \end{align*}
   }
   where we have used that $\mathcal{L}_s \Phi = 0 $ in the second-last equality. 

   We now show that the bulk terms vanish as $k\to \infty$; we shall treat only the first one, as the second is analogous. We have
   \begin{align*}
       \bigg|\int_{M\times (\ep, R)} V \nabla \Phi \cdot \nabla \eta_k z^{1-2s} d\mu dz\bigg|&\leq \frac{C_{\ep,R}}{k}\int_{M\times (\ep, R)} |V| |\nabla \Phi|   z^{\frac{1-2s}{2}} d\mu dz  \\
       &\leq\frac{C_{\ep,R}}{k} [ V ] _{\widehat{H}^1_w(\ov M)} [\Phi]_{\widehat{H}^1_w(\ov M)}\to 0, \quad{\rm as}\,k\to\infty,
   \end{align*}
   where we used that $z\in (\ep,R)$, Cauchy-Schwarz and (the proof of) Lemma \ref{lem: conv V}.

   Hence, letting $k\to \infty$ in \eqref{eq: integrated} gives
\begin{align*}
     R^{1-2s} \int_M (\Phi \partial_z V - V \partial_z \Phi)(\cdot, R) \, d\mu - \ep^{1-2s} \int_M (\Phi \partial_z V - V \partial_z \Phi)(\cdot, \ep) \, d\mu = \int_{M\times (\ep, R)} F \Phi \, d\mu dz  .   
\end{align*}

We claim that, letting $\ep \downarrow 0 $ and $R \uparrow \infty$, the two terms on the l.h.s. involving $V \partial_z \Phi$ vanish. Indeed, since $ \varphi \in C_c^\infty(M)$ we have in particular $ \varphi \in {\rm Dom}((-\Delta)^s)$. Then, by \cite[Theorem 1.1]{Stinga2010} (or by the main results in \cite{CSbanach1, CSbanach2}) we get $   \lim_{\ep\downarrow 0} \ep^{1-2s} \partial_z \Phi(\cdot, \ep) =  \beta_s^{-1}\Lambda^{\hspace{-1pt} s}  \varphi$ in $L^2(M)$. In particular $\ep^{1-2s} \partial_z \Phi(\cdot, \ep)$ is uniformly bounded in $L^2(M)$ as $\ep \downarrow 0$. Moreover, by Lemma \ref{lem: conv V} we have that $\lim_{\ep\downarrow 0} V(\cdot, \ep)=0 $ in $L^2(M)$. Thus 
\begin{equation}\label{eq: conv Stinga}
    \ep^{1-2s} \int_M  V (\cdot, \ep)  \partial_z \Phi (\cdot, \ep) \, d\mu \to 0 ,  \quad \mbox{as } \ep \downarrow 0 .  
\end{equation}
By Cauchy-Schwartz, \eqref{eq: decay d_z U} with $p=2$ and (the proof of) Lemma \ref{lem: conv V}
\begin{align*}
    R^{1-2s} \int_M  |V|(\cdot,R) |\partial_z \Phi| (\cdot, R) \, d\mu&\leq R^{1-2s}\|V(\cdot,R)\|_{L^2(M)}\|\partial_z \Phi (\cdot, R)\|_{L^2(M)} \\
    &\leq C_s R^{s}[V]_{\widehat{H}^1_w(\ov M)} \cdot R^{-2s} \|\varphi\|_{L^2(M)} , 
\end{align*}
and the latter tends to zero when $R\uparrow \infty$.

Hence, letting $\ep \downarrow 0 $ and $R \uparrow \infty$ gives 
\begin{equation*}
     \lim_{z \uparrow \infty} \int_{M} (z^{1-2s} \partial_z V )(\cdot, z) \varphi \, d\mu  - \lim_{z\downarrow 0} \int_{M} (z^{1-2s} \partial_z V )(\cdot, z) \varphi \, d\mu = \int_{\ov M} F\Phi \, d\mu dz .  
\end{equation*}

We now compute the r.h.s. of this equality. Since the fractional Poisson kernel is symmetric $\mathcal{P}_z^{(s)}(x,y)=\mathcal{P}_z^{(s)}(y,x)$, by Fubini's theorem
\begin{align*}
    \int_{\widetilde{M}} F \Phi \, d\mu dz &= \int_{\widetilde{M}} F (\mathcal{P}_z^{(s)} \varphi) \, d\mu dz  =  \int_{\widetilde{M}} (\mathcal{P}_z^{(s)} F ) \varphi \, d\mu dz =  \int_{M} \left( \int_0^\infty \mathcal{P}_z^{(s)} F (\cdot, z) \, dz \right) \varphi \, d\mu . 
\end{align*}
Combining this identity with the one for the l.h.s. and comparing the coefficients of $\varphi$, we obtain 
 \begin{equation*}
       \lim_{z \uparrow \infty} z^{1-2s} \partial_z V (\cdot, z) - \lim_{z\downarrow 0} z^{1-2s} \partial_z V (\cdot, z)  = \int_0^\infty \mathcal{P}_z^{(s)} F (\cdot, z) \, dz  \qquad \mbox{in } \mathscr{D}'(M) .  
    \end{equation*}
     This concludes the proof.
\end{proof}

\section{Proof of the main results}

All proofs in this section follow the same scheme: compare the extension of the target quantity with the corresponding quantity built from extensions, identify the forcing term in the inhomogeneous extension equation, and apply Proposition \ref{prop: rep formula solutions}. 

\subsection{Pointwise fractional Leibniz rule: proof of Theorem \ref{thm: baby Boch Gamma_1}}

\begin{proof}[Proof of Theorem \ref{thm: baby Boch Gamma_1}] Assume first that $u,v \in C^\infty_c(M)$, we will later argue by density to obtain the general case. Let $P$ be the CS extension of $uv \in H^s(M)$. We have 
     \begin{align*}
         \Lambda^{\hspace{-1pt} s} (u v ) & =  \beta_s z^{1-2s} \partial_z P  (\cdot, 0^+)  \\ & =  \beta_s z^{1-2s} \partial_z (U V )  (\cdot, 0^+)  + \beta_s z^{1-2s} \partial_z (P-UV)  (\cdot, 0^+) \\ &=   \beta_s  z^{1-2s} U  \partial_z  V  (\cdot, 0^+) +  \beta_s  z^{1-2s} V  \partial_z  U  (\cdot, 0^+)  + \beta_s z^{1-2s} \partial_z (P-UV)  (\cdot, 0^+) \\ &= u  \Lambda^{\hspace{-1pt} s} v + v  \Lambda^{\hspace{-1pt} s} u  + \beta_s z^{1-2s} \partial_z (P-UV)  (\cdot, 0^+) , 
     \end{align*}
that is 
\begin{equation}\label{eq: Gamma_1 D equation}
    \Gamma_1(u,v) = \frac{\beta_s}{2}   z^{1-2s} \partial_z (P-UV)  (\cdot, 0^+) . 
\end{equation}

Since $ \mathcal{L}_{s} U=0$ and $ \mathcal{L}_{s} V=0$ in $\widetilde M$, it is straightforward that 
\begin{equation*}
     \mathcal{L}_{s} (UV) = U  \mathcal{L}_{s} V  + V  \mathcal{L}_{s} U   + 2 z^{1-2s} ( \widetilde \nabla U \cdot \widetilde \nabla V ) = 2 z^{1-2s} ( \widetilde \nabla U \cdot \widetilde \nabla V ) .
\end{equation*}
Hence, the difference $D:=P-UV \in \widehat{H}^1_w(\ov M)$ solves
    \begin{equation*}
    \begin{cases} 
      \mathcal{L}_{s} D = - 2 z^{1-2s} ( \widetilde \nabla U \cdot \widetilde \nabla V )   &   \mbox{in } \ov M ,  \\ D(\cdot,0) =0 &  \mbox{on } M   . 
    \end{cases}
\end{equation*}
 By Proposition \ref{prop: rep formula solutions}, we get 
\begin{equation*}
          z^{1-2s} \partial_z D (\cdot, 0^+) -   \lim_{z \uparrow \infty } z^{1-2s} \partial_z D (\cdot, z) = 2 \int_0^\infty \mathcal{P}_z^{(s)}  ( \widetilde \nabla U (\cdot, z) \cdot \widetilde \nabla V (\cdot, z) ) z^{1-2s}  dz ,  
    \end{equation*}
    which gives (recalling \eqref{eq: Gamma_1 D equation})
    \begin{equation*}
     \Gamma_1(u,v) -   \frac{\beta_s}{2} \lim_{z \uparrow \infty } z^{1-2s} \partial_z D (\cdot, z) = \beta_s \int_0^\infty \mathcal{P}_z^{(s)}  ( \widetilde \nabla U (\cdot, z) \cdot \widetilde \nabla V (\cdot, z) ) z^{1-2s}  dz .
 \end{equation*}

We are left to show that the boundary term at infinity vanishes, but this easily follows by \eqref{eq: decay d_z U} with $p=1$ applied three times (to $P, U$, and $V$, together with $U, V \in L^\infty(\ov M)$),
\begin{equation*}
    z^{1-2s} \partial_z D = z^{1-2s} \partial_z P - z^{1-2s} (\partial_z U)V -  z^{1-2s} (\partial_z V)U \to 0 , \quad \mbox{in } L^1(M) \mbox{ as } z\uparrow \infty. 
\end{equation*}
This concludes the proof in the case of $u,v \in C^\infty_c(M)$.

 The general case follows by density using Lemma \ref{lem: density in Hs}. Let $u, v \in H^s(M)$ and $u_k, v_k \in C_c^\infty(M)$ be such that $u_k \to u$ and $v_k \to v$ in $H^s(M)$. Let also $U_k $ and $ V_k$ be their CS extensions. By the first part of the proof 
 \begin{equation*}
   \Lambda^{\hspace{-1pt} s}(u_k v_k)-u_k\Lambda^{\hspace{-1pt} s} v_k-v_k\Lambda^{\hspace{-1pt} s} u_k  =  2 \Gamma_1(u_k,v_k) = 2\beta_s \int_0^\infty \mathcal{P}_z^{(s)}  ( \widetilde \nabla U_k (\cdot, z) \cdot \widetilde \nabla V_k (\cdot, z) ) z^{1-2s}  dz.
 \end{equation*}
 Multiplying this equation by $\varphi \in C_c^\infty(M)$, integrating over $M$, and using that $\Lambda^{\hspace{-1pt} s}$ is self-adjoint for every $s\in (0,1)$ gives 
      \begin{align*}
   \langle u_k v_k,  \Lambda^{\hspace{-1pt} s} \varphi \rangle_{L^2}- \langle \Lambda^{\hspace{-1pt} s/2}  (u_k \varphi),  \Lambda^{\hspace{-1pt} s/2} v_k \rangle_{L^2} -\langle \Lambda^{\hspace{-1pt} s/2} (v_k \varphi),&  \Lambda^{\hspace{-1pt} s/2} u_k \rangle_{L^2}    \\ & = 2\beta_s \int_M \int_0^\infty \mathcal{P}_z^{(s)}  ( \widetilde \nabla U_k  \cdot \widetilde \nabla V_k ) \varphi z^{1-2s}  dz d\mu \\ &= 2\beta_s  \int_0^\infty \ns \int_M   \widetilde \nabla U_k  \cdot \widetilde \nabla V_k  (\mathcal{P}_z^{(s)} \varphi) z^{1-2s}  dz d\mu \\ &= 2\beta_s \int_{\widetilde M}    \widetilde \nabla U_k  \cdot \widetilde \nabla V_k  \Phi z^{1-2s}  dz d\mu , 
 \end{align*}
 where we have also used that $\mathcal{P}_z^{(s)}$ is self-adjoint and we have denoted by $\Phi := \mathcal{P}_z^{(s)} \varphi$ the CS extension of $\varphi$. 

 By the strong convergence of $u_k$ and $v_k$ in $H^s(M)$ we have that 
 
 \begin{align*}
    & \Lambda^{\hspace{-1pt} s/2} u_k \to \Lambda^{\hspace{-1pt} s/2} u \,\,  \mbox{ in } L^2(M) , \\  & \Lambda^{\hspace{-1pt} s/2} v_k \to \Lambda^{\hspace{-1pt} s/2} v \,\,   \mbox{ in } L^2(M) , \\ &  \widetilde\nabla U_k \to \widetilde\nabla U \,\, \mbox{ in }  L^2(\widetilde M, z^{1-2s} d\mu dz) ,  \\ &  \widetilde\nabla V_k \to \widetilde\nabla V \,\, \mbox{ in }  L^2(\widetilde M, z^{1-2s} d\mu dz) . 
 \end{align*}
 Moreover, by the maximum principle $\|\Phi\|_{L^\infty(\ov M)} \le \|\varphi\|_{L^\infty(M)}$. These conditions imply\footnote{Together with the fact that multiplication by a fixed $\varphi \in C_c^\infty(M)$ is a continuous operation on $H^s(M)$.} that each term in the previous equality passes to the limit as $k\to \infty$ to give 
  \begin{align*}
   \langle u v,  \Lambda^{\hspace{-1pt} s} \varphi \rangle_{L^2}- \langle \Lambda^{\hspace{-1pt} s/2}  (u \varphi),  \Lambda^{\hspace{-1pt} s/2} v \rangle_{L^2} -\langle \Lambda^{\hspace{-1pt} s/2} (v \varphi), \Lambda^{\hspace{-1pt} s/2} u \rangle_{L^2} =  2\beta_s \int_{\widetilde M}    \widetilde \nabla U  \cdot \widetilde \nabla V  \Phi z^{1-2s}  dz d\mu , 
 \end{align*}
that is
\begin{equation*}
    \langle \Gamma_1(u,v), \varphi \rangle_{L^2} = \int_M \left( \beta_s \int_0^\infty \mathcal{P}_z^{(s)}  ( \widetilde \nabla U \cdot \widetilde \nabla V ) z^{1-2s}  dz\right) \varphi d\mu , 
\end{equation*}
as desired. 

 \end{proof}

\subsection{Fractional Bochner's formula: proof of Theorem \ref{thm: Bochner ricci}}

\begin{lemma}\label{lem: boch extension}
Let $U\in C^\infty(\ov M )$. Then 
\[
\frac12 \mathcal{L}_{s} (|\nabla U|^2)
=
 \nabla (\mathcal{L}_{s} U) \cdot \nabla U 
+
 \big( |\nabla \widetilde \nabla U|^2
+
\Ric (\nabla  U,\nabla  U) \big) z^{1-2s}  , 
\]
where $\Ric$ is the Ricci curvature of $M$. In particular, if $\mathcal{L}_{s} U=0$ in $\ov M$, then
\[
\frac12 \mathcal{L}_{s} (|\nabla U|^2)
=
\big( |\nabla \widetilde \nabla U|^2
+
\Ric(\nabla  U,\nabla  U) \big) z^{1-2s}.
\]
\end{lemma}
\begin{proof}
    The proof is a straightforward application of the classical Bochner's identity to each horizontal slice $M\times \{z\}$. Indeed
    \begin{align*}
         \mathcal{L}_{s} (|\nabla U|^2) &= \widetilde \dive (z^{1-2s} \widetilde \nabla |\nabla U|^2) \\ &= z^{1-2s} \Delta|\nabla U|^2 +  \partial_z(z^{1-2s} \partial_z|\nabla U|^2) \\ &= 2 z^{1-2s} (\nabla U \cdot \nabla \Delta U +|\nabla^2 U|^2 + \Ric(\nabla U, \nabla U)) + 2 \partial_z(z^{1-2s}\nabla U \cdot \nabla \partial_z U ) \\ &= 2 z^{1-2s} \big( \nabla U \cdot \nabla \Delta U + \nabla U \cdot \nabla \partial_z(z^{1-2s}\partial_z U) +  |\nabla^2 U|^2 + |\nabla \partial_z U|^2 + \Ric(\nabla U, \nabla U)  \big) \\ &= 2  \nabla (\mathcal{L}_{s} U) \cdot \nabla U 
+
 2 |\nabla \widetilde \nabla U|^2 z^{1-2s}
+
2 \Ric(\nabla  U,\nabla  U) \big) z^{1-2s} ,  
    \end{align*}
    as desired. 
\end{proof}

 \begin{proof}[Proof of Theorem \ref{thm: Bochner ricci}] We prove \eqref{eq: eq for A} first. By Proposition \ref{prop: rep formula solutions} and Bochner's formula we get (note that this computation is well-defined in the sense of distributions by Lemma \ref{lem: decay at inf} and Lemma \ref{lem: L^1 bounds}) 
 \begin{align*}
     \frac{1}{2}\Delta \Gamma_1(u) & = \frac{\beta_s}{2} \Delta  \left(  \int_0^\infty \mathcal{P}_z^{(s)} \big( |\ov \nabla  U |^2 \big)  z^{1-2s}  dz \right) \\ &= \beta_s     \int_0^\infty \mathcal{P}_z^{(s)} \Big( \frac{1}{2}\Delta|\ov \nabla  U |^2 \Big)  z^{1-2s}  dz  \\ &= \beta_s     \int_0^\infty \mathcal{P}_z^{(s)} \Big( | \nabla^2  U |^2 + \nabla U \cdot \Delta \nabla U + \Ric(\nabla U, \nabla U) + \partial_z U \Delta \partial_z U +|\nabla \partial_z U|^2 \Big)  z^{1-2s}  dz  \\ & = \beta_s    \int_0^\infty \mathcal{P}_z^{(s)} \Big( | \ov \nabla \nabla  U |^2 + \ov \nabla U \cdot  \ov \nabla \Delta U + \Ric(\nabla U, \nabla U) \Big)  z^{1-2s}  dz \\ & = \beta_s    \int_0^\infty \mathcal{P}_z^{(s)} \Big( | \ov \nabla \nabla  U |^2  + \Ric(\nabla U, \nabla U) \Big)  z^{1-2s}  dz  + \Gamma_1(u, \Delta u ) , 
 \end{align*}
 where we have used Proposition \ref{prop: rep formula solutions} again in the last line to infer that 
 \begin{equation*}
     \Gamma_1(u, \Delta u ) = \beta_s \int_0^\infty \mathcal{P}_z^{(s)} \big( \ov \nabla U \cdot  \ov \nabla \Delta U \big) z^{1-2s} dz. 
 \end{equation*}
 Thus 
 \begin{equation*}
     \mathcal{A}(u) = \frac{1}{2}\Delta \Gamma_1(u) -  \Gamma_1(u, \Delta u ) = \beta_s    \int_0^\infty \mathcal{P}_z^{(s)} \Big( | \ov \nabla \nabla  U |^2  + \Ric(\nabla U, \nabla U) \Big)  z^{1-2s}  dz ,  
 \end{equation*}
which gives \eqref{eq: eq for A}.

 The proof of \eqref{eq: eq for B} is a second-order version of the proof of \eqref{eq: Gamma_1 rep} and follows a similar scheme. Let $P$ be the CS extension of $|\nabla u|^2$ to $\widetilde M $. We have 
     \begin{align*}
           \Lambda^{\hspace{-1pt} s} (|\nabla u|^2) & = \beta_s z^{1-2s}\partial_z P  (\cdot, 0^+)  \nonumber \\ & = \beta_s z^{1-2s}\partial_z (|\nabla U|^2)  (\cdot, 0^+)  + \beta_s z^{1-2s}\partial_z (P-|\nabla U|^2)  (\cdot, 0^+) \nonumber \\ &=  2  \nabla U \cdot (\beta_s z^{1-2s}\partial_z  \nabla U)  (\cdot, 0^+)  +\beta_s z^{1-2s} \partial_z (P - |\nabla U|^2)  (\cdot, 0^+)\nonumber \\ &= 2 \nabla u \cdot \nabla  \Lambda^{\hspace{-1pt} s}  u  + \beta_s z^{1-2s} \partial_z (P-|\nabla U|^2)  (\cdot, 0^+) ,
     \end{align*}
    that is 
    \begin{equation}\label{eq: Bu D equation}
     \mathcal{B} (u) = \frac{\beta_s}{2}  z^{1-2s} \partial_z (P - |\nabla U|^2)  (\cdot, 0^+) . 
\end{equation}

Since $\mathcal{L}_{s}U = 0 $, Lemma \ref{lem: boch extension} gives
\begin{align*}
     \frac{1}{2} \mathcal{L}_{s} (|\nabla U|^2)  = z^{1-2s} \big( |\widetilde \nabla \nabla U|^2 + \Ric(\nabla U , \nabla U) \big)  . 
\end{align*}
Hence, the difference $D:=P-|\nabla U|^2 \in \widehat{H}^1_w(\ov M)$ solves
    \begin{equation*}
    \begin{cases} 
     \mathcal{L}_{s} D = - 2 z^{1-2s} \big( |\widetilde \nabla \nabla U|^2 + \Ric(\nabla U , \nabla U) \big)    &   \mbox{in } \widetilde M ,  \\ D(\cdot,0)=0 &  \mbox{on } M  . 
    \end{cases}
\end{equation*}
Observe that the forcing term in the previous equation is in $L^1(\widetilde{M})$ thanks to \eqref{eq: L^1 bound d_z Du} and our assumptions; we can therefore apply Proposition \ref{prop: rep formula solutions} to get
 \begin{align*}
         z^{1-2s} \partial_z D (\cdot, 0^+) -  \lim_{z \uparrow \infty } z^{1-2s} \partial_z D (\cdot, z) & = 2  \int_0^\infty \ns \mathcal{P}^{(s)}_z \Big( |\widetilde \nabla \nabla U(\cdot, z)|^2 +  \Ric(\nabla U (\cdot, z) , \nabla U (\cdot, z) )  \Big) z^{1-2s}  dz,
    \end{align*}
and the latter gives (recalling \eqref{eq: Bu D equation})
\begin{equation}\label{eq: boch final proof}
     \mathcal{B} (u)-  \frac{\beta_s}{2} \lim_{z \uparrow \infty } z^{1-2s} \partial_z D (\cdot, z)  =  \beta_s \int_0^\infty \ns \mathcal{P}^{(s)}_z \Big( |\widetilde \nabla \nabla U(\cdot, z)|^2 +  \Ric(\nabla U (\cdot, z) , \nabla U (\cdot, z) )  \Big) z^{1-2s}  dz . 
\end{equation}

We are left to show that the boundary term at infinity vanishes. For $z>0$ we have 
\begin{equation*}
    z^{1-2s} \partial_z D  =  z^{1-2s} \partial_z P  - z^{1-2s} \partial_z|\nabla U|^2 = z^{1-2s} \partial_z P  - 2 z^{1-2s} \nabla U \cdot \partial_z \nabla U . 
\end{equation*}
As $z\uparrow \infty$, the first term tends to zero in $L^1(M)$ by \eqref{eq: decay d_z U} with $p=1$. The second term also tends to zero in $L^1(M)$ by Cauchy-Schwarz and \eqref{eq: decay DU}, \eqref{eq: decay d_z DU} since $ u \in H^1(M)$. 

Hence, letting $z\uparrow \infty$ in \eqref{eq: boch final proof} concludes the proof.

 \end{proof}

 We can now prove the formula for $\Gamma_2$ in Corollary \ref{thm: Gamma_2}. 

 \begin{proof}[Proof of Corollary \ref{thm: Gamma_2}] Let us compute 
 \begin{align*}
     \beta_s \int_0^\infty &  \mathcal{P}_z^{(s)} \big( \mathcal{B}(U) \big)  z^{1-2s} dz \\ & = \beta_s \int_0^\infty \mathcal{P}_z^{(s)} \Big( \frac{1}{2}  \Lambda^{\hspace{-1pt} s}|\nabla U|^2 - \nabla U \cdot \nabla \Lambda^{\hspace{-1pt} s} U \Big)  z^{1-2s} dz \\ &= \beta_s \int_0^\infty \mathcal{P}_z^{(s)} \Big( \frac{1}{2}  \Lambda^{\hspace{-1pt} s}|\ov \nabla U|^2 - \frac{1}{2}  \Lambda^{\hspace{-1pt} s}|\partial_z U|^2 - \ov \nabla U \cdot \ov \nabla \Lambda^{\hspace{-1pt} s} U + \partial_z U \partial_z \Lambda^{\hspace{-1pt} s} U \Big)  z^{1-2s} dz \\ &= \frac{1}{2} \Lambda^{\hspace{-1pt} s} \Gamma_1(u) - \Gamma_1(u, \Lambda^{\hspace{-1pt} s} u) - \beta_s  \int_0^\infty \mathcal{P}_z^{(s)} \Big( \frac{1}{2}  \Lambda^{\hspace{-1pt} s}|\partial_z U|^2 - \partial_z U  \Lambda^{\hspace{-1pt} s} \partial_z U \Big)  z^{1-2s} dz  ,
 \end{align*}
where we have used that $ \Lambda^{\hspace{-1pt} s} U $ is the extension of $\Lambda^{\hspace{-1pt} s} u$.

 Moreover, by the very definition of $\Gamma_1$, for every fixed $z>0$ there holds
 \begin{equation*}
     \frac{1}{2}  \Lambda^{\hspace{-1pt} s}|\partial_z U|^2 - \partial_z U  \Lambda^{\hspace{-1pt} s} \partial_z U = \Gamma_1( \partial_z U ) , 
 \end{equation*}
 hence
 \begin{align*}
     \beta_s \int_0^\infty   \mathcal{P}_z^{(s)} \big( \mathcal{B}(U) \big)  z^{1-2s} dz & = \frac{1}{2} \Lambda^{\hspace{-1pt} s} \Gamma_1(u) - \Gamma_1(u, \Lambda^{\hspace{-1pt} s} u) - \beta_s \int_0^\infty   \mathcal{P}_z^{(s)} \big( \Gamma_1( \partial_z U )\big)z^{1-2s} dz \\ &= \Gamma_2(u)  - \beta_s \int_0^\infty   \mathcal{P}_z^{(s)} \big( \Gamma_1( \partial_z U )\big)z^{1-2s} dz , 
 \end{align*}
 which is what we wanted.
     
 \end{proof}

\subsection{Remainder in the Córdoba-Córdoba inequality: proof of Theorem \ref{thm: C-C general phi}}

 \begin{proof}[Proof of Theorem \ref{thm: C-C general phi}] Observe that with our hypothesis (in particular $\phi(0)=0$ whenever $M$ is non-compact) we have $\phi(u) \in H^s(M)$. Denote by $U_\phi$ the CS extension of $\phi(u)$. We have 
     \begin{align*}
          \Lambda^{\hspace{-1pt} s} (\phi(u)) & =  \beta_s   z^{1-2s} \partial_z U_\phi  (\cdot, 0^+)  \\ & = \beta_s   z^{1-2s} \partial_z (\phi(U))  (\cdot, 0^+)  + \beta_s   z^{1-2s} \partial_z (U_\phi-\phi(U))(\cdot, 0^+) \\ &=   \beta_s  z^{1-2s} \phi'(U) \partial_z  U  (\cdot, 0^+)  + \beta_s   z^{1-2s} \partial_z (U_\phi-\phi(U))  (\cdot, 0^+) \\ &= \phi'(u)  \Lambda^{\hspace{-1pt} s} u  + \beta_s   z^{1-2s} \partial_z (U_\phi-\phi(U))  (\cdot, 0^+) . 
     \end{align*}

Since $\mathcal{L}_{s}U=0$, a direct computation shows 
\begin{align*}
     \mathcal{L}_{s}(\phi(U)) & = \ov \dive (z^{1-2s} \phi'(U) \ov \nabla U ) \\ & = \phi'(U) \mathcal{L}_{s}(U) +  z^{1-2s}\phi''(U)|\widetilde \nabla U|^2 = z^{1-2s}\phi''(U)|\widetilde \nabla U|^2  . 
\end{align*}
Hence, the difference $D:=U_\phi-\phi(U)  \in \widehat{H}^1_w(\ov M) \cap L^\infty(\ov M) $ solves
    \begin{equation*}
    \begin{cases} 
     \mathcal{L}_{s}D = - z^{1-2s}\phi''(U)|\widetilde \nabla U|^2    &   \mbox{in } \ov M ,   \\ D(\cdot,0)=0 &  \mbox{on } M  . 
    \end{cases}
\end{equation*}
By Proposition \ref{prop: rep formula solutions}
\begin{equation*}
         z^{1-2s} \partial_z D (\cdot, 0^+) -  \lim_{z \uparrow \infty } z^{1-2s} \partial_z D (\cdot, z) = \int_0^\infty\mathcal{P}^{(s)}_z\big( \phi''(U)|\widetilde \nabla U|^2(\cdot, z) \big) z^{1-2s}  dz ,  
    \end{equation*}
    which gives
    \begin{equation*}
     \Lambda^{\hspace{-1pt} s}(\phi(u)) - \phi'(u)  \Lambda^{\hspace{-1pt} s} u  -  \lim_{z \uparrow \infty } z^{1-2s} \partial_z D (\cdot, z) = \beta_s  \int_0^\infty \mathcal{P}^{(s)}_z\big( \phi''(U)|\widetilde \nabla U|^2(\cdot, z) \big) z^{1-2s}  dz  . 
 \end{equation*}

 Lastly, again the boundary term at infinity vanishes since by Lemma \ref{lem: decay at inf} applied to $U_\phi$ and to $U$ (together with $U \in L^\infty(\ov M)$ and $\phi'$ continuous)
\begin{equation*}
    z^{1-2s} \partial_z D = z^{1-2s} \partial_z U_\phi - z^{1-2s} \phi'(U) \partial_z U  \to 0 , \quad \mbox{in } L^1(M) \mbox{ as } z\uparrow \infty. 
\end{equation*}
This concludes the proof.
 \end{proof}

 \begin{proof}[Proof of Theorem \ref{thm: Kato}] Assume that $u\neq{\rm const.}$, otherwise the thesis is trivially true. Let $\varphi\in {\rm Lip}_c^\infty(M)$ and w.l.o.g. assume $\varphi\geq 0$: indeed one can write $\varphi=\varphi^+-\varphi^-$ and argue separately for the positive and the negative part. Let now $\ep>0$ and $\phi_\ep(t)=\sqrt{t^2+\ep^2}-\ep$ so that, as $\ep \to 0$, $  \phi_\ep' \to {\rm sgn}(\cdot)  $ pointwise everywhere and $ \phi_\ep'' \to 2 \delta_0 $ weakly as Radon measures. Applying Theorem \ref{thm: C-C general phi} to $\phi_\ep(t)$ and integrating against $\varphi$ we get
 \begin{equation}\label{eq: Kato eq before limit}
      \int_M\Big(\Lambda^{\hspace{-1pt} s}\phi_\ep(u) - \phi_\ep^\prime(u) \Lambda^{\hspace{-1pt} s} u\Big)\varphi \, d\mu = \beta_s \int_M\int_0^\infty \varphi \mathcal{P}^{(s)}_z \left(  \phi^{\prime\prime}_\ep(U) | \ov \nabla U(\cdot,z)|^2 \right) z^{1-2s} dz d\mu .
 \end{equation}
We first compute the limit of the r.h.s. of this equation, which is the nontrivial part.  

Let $\Phi:= \mathcal{P}^{(s)}_z \varphi $ be the CS extension of $\varphi$. Exploiting that $\mathcal{P}^{(s)}_z$ is self-adjoint together with Coarea formula, we get
 \begin{align*} \beta_s \int_M \int_0^\infty  \varphi\mathcal{P}^{(s)}_z \left( \phi^{\prime\prime}_\ep(U) | \ov \nabla U(\cdot,z)|^2 \right)  z^{1-2s} dz d\mu   &=
     \beta_s   \int_0^\infty \ns \int_M    \phi^{\prime\prime}_\ep(U)  | \ov \nabla U|^2 \Phi z^{1-2s}  d\mu dz \\ & = \beta_s \int_{\R} \phi^{\prime\prime}_\ep(t) \left( \int_{\{U=t \}}  |\ov \nabla U| \Phi z^{1-2s}d\mathcal{H}^{n} \right) dt  \\ &= \int_{\R} \phi^{\prime\prime}_\ep(t) F_\varphi(t) \, dt , 
 \end{align*}
 where we defined 
 \begin{equation*}
     F_\varphi(t) :=  \beta_s \int_{\{U=t \}}  |\ov \nabla U| \Phi z^{1-2s}d\mathcal{H}^{n}(x,z)\quad{\rm for}\;{\rm a.e.}\;t\in\R.
 \end{equation*}
From now on, with a little bit of abuse of notation, we shall denote by $t\mapsto F_\varphi(t)$ the everywhere defined representative of the $L^1(\R)$ function $F_\varphi$. Observe that such a representative is well-defined for every $t\in \R$, since the integrand is nonnegative and crucially, since the operator $\mathcal{L}_s$ is locally uniformly elliptic in $\ov M$ and $u, U \not\equiv {\rm const}$, by \cite[Theorem 1.7 and Theorem 1.10]{HardtSimon} the level set $\{U=t\}$ is locally in $\ov M$ the (at most) countable union of $C^1$ submanifolds with finite $\mathcal{H}^n$-measure, outside a singular set with Hausdorff dimension $\le n-1$. In particular $\{U=t\}$ is $\sigma$-finite w.r.t. $\mathcal{H}^n$ in $\ov M$. 

First we claim that, for every $t\in \R$, $F_\varphi(t)$ can be written as 
\begin{align}
    F_\varphi(t)  & = - \beta_s \int_{\{U>t\}}  \ov \nabla U \cdot \ov \nabla \Phi z^{1-2s} d\mu dz - \int_{\{ u>t\}} (\Lambda^{\hspace{-1pt} s} u ) \varphi \, d\mu - \int_{\{u=t\} \cap \{ \Lambda^s u > 0 \}} (\Lambda^{\hspace{-1pt} s} u ) \varphi \, d\mu \label{eq: f(t) as sum2} \\ &=:  H(t)+G(t)+R(t) .  \nonumber 
\end{align}
To show this, we first localize with cutoffs very similarly to the proof of Proposition \ref{prop: rep formula solutions}. Fix $p \in M$ and let $\eta_k \in C_c^\infty(M)$ be a standard cutoff with $\eta_k =1$ in $B_k(p)$ and $\eta_k =0$ in $ M \setminus B_{2k}(p)$. Let also $0<\ep < R$ and $I_{\ep,R}:= (\ep,R)$.

Indeed, on the one hand since $\mathcal{L}_s U = 0 $ in $\ov M$ we have 
\begin{align*}
   \int_{\{U>t\} \cap I_{\ep,R} } \ns \widetilde \dive  (\ov \nabla U \Phi \eta_k & z^{1-2s})  \, d\mu dz   \\ & =  \int_{\{U>t\} \cap I_{\ep,R} } \ns \ov \nabla U \cdot \ov \nabla (\Phi \eta_k) z^{1-2s} d\mu dz \\ &=  \int_{\{U>t\}  \cap I_{\ep,R} } \ns \ov \nabla U \cdot \ov \nabla \Phi \eta_k z^{1-2s} d\mu dz + \int_{\{U>t\}  \cap I_{\ep,R} } \ns \nabla U \cdot \nabla \eta_k  \Phi  z^{1-2s} d\mu dz . 
\end{align*}
On the other hand, by the divergence theorem
\begin{align*}
       \int_{\{U>t\} \cap  I_{\ep,R}} \ns  \widetilde \dive(\ov \nabla U & \Phi \eta_k z^{1-2s}) d\mu dz  \\ & =  \int_{ \partial (\{U>t\} \cap  I_{\ep,R} )} \nu \cdot  \ov \nabla U \Phi \eta_k z^{1-2s} d\mathcal{H}^n \\ & = -  \int_{\{U=t\} \cap  I_{\ep,R} } \frac{\ov \nabla U}{|\ov \nabla U|} \cdot \ov \nabla U  \Phi \eta_k z^{1-2s} d\mathcal{H}^n -  \int_{\{U(\cdot, \ep)>t\}} \ep^{1-2s} \partial_z U \Phi (\cdot, \ep)\eta_k \, d\mu \\ & \hspace{12pt} +  \int_{\{U(\cdot, R)>t\}} R^{1-2s} \partial_z U \Phi (\cdot, R) \eta_k \, d\mu \\ & = - \int_{\{U=t\} \cap  I_{\ep,R} } |\ov \nabla U|  \Phi \eta_k z^{1-2s} d\mathcal{H}^n -  \int_{\{U(\cdot, \ep)>t\}} \ep^{1-2s} \partial_z U \Phi(\cdot, \ep) \eta_k \, d\mu \\ & \hspace{12pt} +  \int_{\{U(\cdot, R)>t\}} R^{1-2s} \partial_z U \Phi(\cdot, R)\eta_k \, d\mu .
\end{align*}
Letting $k\to \infty$ and using that
\begin{align*}
   & \int_{\{U>t\}  \cap I_{\ep,R} } \ns \ov \nabla U \cdot \ov \nabla \Phi \eta_k z^{1-2s} d\mu dz  \to \int_{\{U>t\}  \cap I_{\ep,R} } \ns \ov \nabla U \cdot \ov \nabla \Phi z^{1-2s} d\mu dz  \quad \mbox{since } U,\Phi \in \widehat{H}^1_w(\ov M) , \\ &  \int_{\{U>t\}  \cap I_{\ep,R} } \ns \nabla U \cdot \nabla \eta_k  \Phi  z^{1-2s} d\mu dz  \to 0 \quad \mbox{since } \nabla U \Phi z^{1-2s} \in L^1(M\times I_{\ep,R}) \mbox{ and } \|\nabla \eta_k\|_{L^\infty} \to 0, \\  & \int_{\{U=t\} \cap  I_{\ep,R} } |\ov \nabla U|  \Phi \eta_k z^{1-2s} d\mathcal{H}^n  \to \int_{\{U=t\} \cap  I_{\ep,R} } |\ov \nabla U|  \Phi z^{1-2s} d\mathcal{H}^n \quad \mbox{by monotone convergence}, \\ & \int_{\{U(\cdot, z)>t\}} z^{1-2s} \partial_z U \Phi \eta_k \, d\mu  \to \int_{\{U(\cdot, z)>t\}} z^{1-2s} \partial_z U \Phi \, d\mu \quad \mbox{by } \eqref{eq: decay d_z U} \mbox{ and dominated convergence},
\end{align*}
we get
\begin{align*}
    \int_{\{U>t\}  \cap I_{\ep,R} } \ns \ov \nabla U \cdot \ov \nabla \Phi z^{1-2s} d\mu dz & = - \int_{\{U=t\} \cap  I_{\ep,R} } |\ov \nabla U|  \Phi z^{1-2s} d\mathcal{H}^n  -  \int_{\{U(\cdot, \ep)>t\}} \ep^{1-2s} \partial_z U \Phi(\cdot, \ep)  \, d\mu \\ & \hspace{12pt} +  \int_{\{U(\cdot, R)>t\}} R^{1-2s} \partial_z U \Phi(\cdot, R) \, d\mu .
\end{align*}
Now we send $R\uparrow \infty$. For the boundary term at $z=R$, using that $\|\Phi(\cdot, R)\|_{L^\infty(M)} \le \|\varphi\|_{L^\infty(M)}$ and \eqref{eq: decay d_z U}, we get  
\begin{equation*}
    \left| \int_{\{U(\cdot, R)>t\}} R^{1-2s} \partial_z U \Phi (\cdot,R) \, d\mu \right| \le  \int_{M} R^{1-2s} |\partial_z U (\cdot,R)| |\Phi (\cdot,R)| \, d\mu  \to 0 . 
\end{equation*}
The first term on the r.h.s. converges as well as $R\uparrow\infty$ thanks to the monotone convergence theorem, and the l.h.s. also converges since $U,\Phi \in \widehat{H}^1_w(\ov M)$. Hence, letting $R \uparrow \infty$ first and then $\ep \downarrow 0$ we are left with 
\begin{equation*}
    \beta_s \int_{\{U>t\}} \ov \nabla U \cdot \ov \nabla \Phi z^{1-2s} d\mu dz = - \beta_s \int_{\{U=t\}} |\ov \nabla U|  \Phi z^{1-2s} d\mathcal{H}^n - \lim_{\ep \downarrow 0 } \beta_s \int_{\{U(\cdot, \ep)>t\}} \ep^{1-2s} \partial_z U \Phi (\cdot, \ep) \, d\mu . 
\end{equation*}
The delicate part is to compute the limit as $\ep\downarrow 0 $ on the r.h.s.  

First, observe that since $\beta_s z^{1-2s}\partial_z U (\cdot, z) \to \Lambda^{\hspace{-1pt} s} u $ in $L^2(M)$ (see the few lines above \eqref{eq: conv Stinga} in the proof of Proposition \ref{prop: rep formula solutions}) and  $\Phi(\cdot,z) \to \varphi$ in $L^2(M)$, points in $\{\Lambda^{\hspace{-1pt} s} u = 0 \}$ give no contribution to the limit, that is 
\begin{equation*}
    \lim_{z \downarrow 0 }  \int_{M} (z^{1-2s} \partial_z U) \mathbbm{1}_{\{U(\cdot, z)>t\}}\Phi \, d\mu = \lim_{z \downarrow 0 } \int_{M} (z^{1-2s} \partial_z U) \mathbbm{1}_{\{U(\cdot, z)>t\} \cap \{\Lambda^{\hspace{-1pt} s} u \neq 0 \} }\Phi \, d\mu .
\end{equation*}
Then, as $z\downarrow 0$, we claim that  
\begin{equation*}
     \mathbbm{1}_{ \{U(\cdot, z)>t\} \cap \{\Lambda^{\hspace{-1pt} s} u \neq 0 \} } \to  \mathbbm{1}_{\{u>t\}} + \mathbbm{1}_{\{u=t\} \cap \{ \Lambda^{\hspace{-1pt} s} u > 0 \}} =  \mathbbm{1}_{\{u>t\} \cup (\{u=t\} \cap \{ \Lambda^{\hspace{-1pt} s} u > 0\})}  \qquad \mbox{in } L^1_{\rm loc}(M) . 
\end{equation*}
It clearly suffices to show almost everywhere convergence in a compact set $K\subseteq M$. We know that 
\begin{equation}\label{eq: expansion U}
    U(x,z) = u(x)+ \frac{z^{2s}}{2s\beta_s} \Lambda^{\hspace{-1pt} s} u (x) + o(z^{2s}) , 
\end{equation}
uniformly for $x\in K$. The fact that the remainder in \eqref{eq: expansion U} is uniform in compact sets is an easy consequence of the semigroup formula \eqref{eq: U formula with heat} and dominated convergence. Now for every $x \in \{U(\cdot, z)>t\} \cap \{\Lambda^{\hspace{-1pt} s} u \neq 0 \}$ we distinguish two cases. 

If $x \in \{ u \neq t\}$ then the expansion \eqref{eq: expansion U} gives 
\begin{equation*}
    {\rm sgn} (U(x,z) - t ) = {\rm sgn} ( u(x) - t  ) ,  
\end{equation*}
for $z$ small uniformly in $x\in K$. Hence, in this case
\begin{equation*}
    \mathbbm{1}_{\{U(\cdot, z)>t\}}(x) = \mathbbm{1}_{\{u>t\}}(x) , \quad \mbox{for } z \mbox{ small.}
\end{equation*}

If $x \in \{ u = t\}$ then \eqref{eq: expansion U} gives 
\begin{equation*}
    {\rm sgn} (U(x,z) - t ) = {\rm sgn} (\Lambda^{\hspace{-1pt} s} u (x)  ) ,  
\end{equation*}
for $z$ small uniformly for $x\in K$. Hence, in this other case 
\begin{equation*}
    \mathbbm{1}_{\{U(\cdot, z)>t\}}(x) = \mathbbm{1}_{\{u=t\} \cap \{ \Lambda^{\hspace{-1pt} s} u > 0 \}}(x) , \quad \mbox{for } z \mbox{ small.}
\end{equation*}
Since exactly one of these two cases holds for every $x\in K$, we have shown the desired convergence of characteristic functions. 

Thus 
\begin{align*}
   \lim_{z \downarrow 0 }   \beta_s \int_{M} (z^{1-2s} \partial_z U) \mathbbm{1}_{\{U(\cdot, z)>t\} \cap \{\Lambda^{\hspace{-1pt} s} u \neq 0 \} }\Phi \, d\mu  = \int_{\{u>t\}} (\Lambda^{\hspace{-1pt} s} u)\varphi \, d\mu  +  \int_{\{u=t\}\cap \{\Lambda^{\hspace{-1pt} s} u > 0 \}} (\Lambda^{\hspace{-1pt} s} u)\varphi \, d\mu . 
\end{align*}
Rearranging the terms, we get \eqref{eq: f(t) as sum2}. Now we conclude the proof using \eqref{eq: f(t) as sum2}.  

A simple application of the Coarea formula gives that $t\mapsto H(t)$ is continuous over all $\R$, hence 
\begin{equation*}
     \lim_{\ep\to 0^+}\int_{\R}\phi_\ep^{\prime\prime}(t)H(t)\,dt = 2H(0) = - \beta_s \int_{\{U>0\}}  \ov \nabla U \cdot \ov \nabla \Phi z^{1-2s} d\mu dz . 
\end{equation*}

Concerning $G$, exchanging the order of integration gives
\begin{align*}
    \int_{\R}\phi_{\ep}^{\prime\prime}(t)G(t) \, dt &= -\int_M(\Lambda^{\hspace{-1pt} s} u)(x)\varphi(x)\left( \int_{-\infty}^{u(x)}\phi_\ep^{\prime\prime}(t)\,dt \right) d\mu(x) \\
    &=-\int_M(\Lambda^{\hspace{-1pt} s} u)(x)\varphi(x)\Big(\phi_\ep^\prime(u(x))+1\Big) d\mu(x) \\ 
    &=-\int_{\{u\neq 0\}}(\Lambda^{\hspace{-1pt} s} u)(x)\varphi(x)\Big(\phi_\ep^\prime(u(x))+1\Big)d\mu(x)-\int_{\{u=0\}}(\Lambda^{\hspace{-1pt} s} u)(x)\varphi(x) \, d\mu(x),
\end{align*}
where we also used $\phi_\ep^\prime(0)=0$. 
Now observe that $\phi_\ep^\prime(u(x))$ converges pointwise to  ${\rm sgn}(u)$ on the set $\{u\neq 0\}$, so that taking the limit yields
\begin{equation*}
 \lim_{\ep\to 0}\int_{\R}\phi_{\ep}^{\prime\prime}(t)G(t)dt=-2\int_{\{u>0\}}(\Lambda^{\hspace{-1pt} s} u) \varphi \, d\mu-\int_{\{u=0\}}(\Lambda^{\hspace{-1pt} s} u) \varphi \, d\mu.
\end{equation*}

Lastly, by Sard's theorem $\mathcal{H}^n(\{u=t\}) = 0 $ for a.e. $t\in \R$, thus $R(t)=0$ for a.e. $t\in \R$ and it does not contribute to the integration against $\phi_\ep''$.

All in all, for the r.h.s. of \eqref{eq: Kato eq before limit}, using \eqref{eq: f(t) as sum2} at $t=0$ we get 
\begin{align*}
    \lim_{\ep\to 0^+}\int_{\R}\phi_\ep^{\prime\prime}(t)F_\varphi(t)\,dt&= 2H(0)  -2\int_{\{u>0\}}(\Lambda^{\hspace{-1pt} s} u)\varphi \, d\mu -\int_{\{u=0\}}(\Lambda^{\hspace{-1pt} s} u)\varphi \, d\mu \\ & = -2\beta_s\int_{\{U>0\}} \ov \nabla U \cdot \ov \nabla \Phi z^{1-2s} d\mu dz-2\int_{\{u>0\}}(\Lambda^{\hspace{-1pt} s} u) \varphi \, d\mu -\int_{\{u=0\}}(\Lambda^{\hspace{-1pt} s} u) \varphi \, d\mu \\
    &=2F_\varphi(0) + 2 \int_{\{u=0\} \cap \{\Lambda ^s u > 0\}} (\Lambda^{\hspace{-1pt} s} u) \varphi \, d\mu  - \int_{\{u=0\}}(\Lambda^{\hspace{-1pt} s} u) \varphi \, d\mu  \\
    &=2F_\varphi(0) +  \int_{\{u=0\} \cap \{\Lambda ^s u > 0\}} (\Lambda^{\hspace{-1pt} s} u) \varphi \, d\mu  - \int_{\{u=0\} \cap \{\Lambda ^s u \le 0\}}(\Lambda^{\hspace{-1pt} s} u) \varphi \, d\mu \\ & = 2F_\varphi(0) +  \int_{\{u=0\}} |\Lambda^{\hspace{-1pt} s} u| \varphi \, d\mu , 
\end{align*}
as desired.

Lastly, the l.h.s. of \eqref{eq: Kato eq before limit} clearly tends to the desired limit by self adjointness of $\Lambda^{\hspace{-1pt} s}$ and the fact that $\phi_\ep(u) \to |u|$ in $L^1(M)$ for the first term, while for the second one we just apply dominated convergence theorem.
\end{proof}

 \begin{proof}[Proof of Proposition \ref{prop:SVI}] Integrating the pointwise identity \eqref{eq: pointwise SV} over $M$ and using that $\int_{M} \Lambda^s(|u|^q) \, d\mu = 0 $ gives
     \begin{align*}
         \int_{M} (|u|^{q-2}u) (-\Delta)^s u \,d\mu & = \frac{4(q-1)}{q^2} \beta_s
\int_M \int_0^\infty
\mathcal{P}_z^{(s)} \left(|\widetilde\nabla(|U|^{q/2})|^2(\cdot,z)\right) z^{1-2s}\,dz d\mu \\ &= \frac{4(q-1)}{q^2} \beta_s
 \int_0^\infty \ns \int_M|\widetilde\nabla(|U|^{q/2})|^2 z^{1-2s}\, d\mu dz  , 
     \end{align*}
     where we have also used
     \begin{equation*}
         \int_M \mathcal{P}_z^{(s)}(x,y) \, d\mu(y) = 1 \qquad \forall \, x\in M,\;z\in(0,\infty).
     \end{equation*}
     Since the $H^s$-seminorm of $|u|^{q/2}$ is the infimum of the ${\widehat H}^1_w (\widetilde M)$ seminorm over all the extension with trace $|u|^{q/2}$, and since $|U|^{q/2}$ is one such extension, we have 
\begin{equation*}
   \beta_s \int_0^\infty \ns \int_M|\widetilde\nabla(|U|^{q/2})|^2 z^{1-2s}\, d\mu dz \ge \int_M |(-\Delta)^{s/2}(|u|^{q/2})|^2 \, d\mu ,
\end{equation*}
which gives what we wanted. 
 \end{proof}

\appendix
 \section{}\makeatletter\def\@currentlabel{Appendix}\makeatother

 \begin{lemma}\label{lem: conv V}
Let $s \in (0,1)$ and let $V \in \widehat{H}^1_w(\widetilde M)$ be smooth such that \(V(\cdot,0)=0\) on $M$. Then  
\[
 V(\cdot,z) \to 0 \, \mbox{ strongly in } L^2(M), \mbox{ as } z \downarrow 0 . 
\]
\end{lemma}

\begin{proof}
Since \(V(\cdot,0)=0\) and \(V\) is smooth up to \(z=0\) we can write  $V(x,z)=\int_0^z \partial_t V(x,t)\,dt$. By Cauchy-Schwarz 
\[
|V(x,z)|^2
\le
\left(\int_0^z t^{2s-1}\,dt\right)
\left(\int_0^z |\partial_tV(x,t)|^2 t^{1-2s} dt\right)
=
\frac{z^{2s}}{2s}\int_0^z |\partial_tV(x,t)|^2 t^{1-2s} dt , 
\]
and integrating over \(M\) gives 
\[
\|V(\cdot,z)\|_{L^2(M)}^2
\le
\frac{z^{2s}}{2s}
\int_M\int_0^z |\partial_tV(x,t)|^2 t^{1-2s} dt d\mu(x) \le \frac{z^{2s}}{2s} \int_{\ov M} |\ov \nabla V|^2 t^{1-2s} d\mu(x)dt  \to 0 ,
\]
as $z \downarrow 0 $. 
\end{proof}

\begin{lemma}\label{lem: decay at inf} 
    Let $s\in (0,1)$, $p\in[1,\infty]$, $u \in C_c^\infty(M)$ and $U$ be its CS extension. Then 
    \begin{align}
        z^{1-2s} \|\partial_z U (\cdot,z) \|_{L^p(M)} & \leq C_s  z^{-2s} \|u\|_{L^p(M)}  , \label{eq: decay d_z U} 
    \end{align}
    and 
    \begin{align}
        z^{1-2s} \|\partial_z \nabla U (\cdot,z) \|_{L^2(M)} & \leq C_s  z^{-2s} \|\nabla  u\|_{L^2(M)}  \label{eq: decay d_z DU}  \\  \| \nabla U (\cdot,z) \|_{L^2(M)} & \leq C_s   \|\nabla u \|_{L^2(M)}  \label{eq: decay DU}
    \end{align}
\end{lemma}
\begin{proof} 
The proof of all these estimates is a direct application of the explicit formula (see \cite[Theorem 1.1]{StingaTorrea})
     \begin{equation}\label{eq: U formula with heat}
         U(x,z) = \frac{z^{2s}}{4^s \Gamma(s)} \int_0^\infty (P_t u)(x) e^{-\frac{z^2}{4t}} \frac{dt}{t^{1+s}} , 
     \end{equation}
     where $P_t$ is the heat semigroup. Indeed, for \eqref{eq: decay d_z U}, by Minkowski's integral inequality
     \begin{align*}
         z^{1-2s} \|\partial_z U(\cdot,z)\|_{L^p(M)}  & \le  \int_{0}^\infty  \| P_t u \|_{L^p(M)} \left|2s- \frac{z^2}{2t}\right| e^{-\frac{z^2}{4t}} \frac{dt}{t^{1+s}}  \\ & \le \|  u \|_{L^p(M)} \int_{0}^\infty  \left|2s- \frac{z^2}{2t}\right| e^{-\frac{z^2}{4t}} \frac{dt}{t^{1+s}} \\ & =  \|u\|_{L^p(M)} \cdot 4^s z^{-2s} \int_0^\infty |2s-2r| e^{-r}r^{s-1} dr \\ & = C_s z^{-2s} \|u\|_{L^p(M)}  , 
     \end{align*}
     where we have used the $L^p$-contractivity of the heat semigroup and made a change of variables.

     The proofs of \eqref{eq: decay d_z DU} and \eqref{eq: decay DU} follow an identical scheme using that $\|\nabla P_t u \|_{L^2(M)} \le \|\nabla u\|_{L^2(M)}$ whenever $u\in H^1(M) = {\rm Dom}((-\Delta)^{1/2})$. 
\end{proof}

\begin{lemma}\label{lem: L^1 bounds}Let $s\in (0,1)$, $u\in C_c^\infty(M)$ and $U$ be its CS extension. Then 
\begin{equation}\label{eq: L^1 bound d_z Du}
    z^{1-2s}  |\nabla \partial_z U|^2  \in L^1(\ov M) . 
\end{equation}
Moreover, if in addition $z^{1-2s}\big(|\nabla^2 U|^2+\Ric(\nabla U, \nabla U) \big) \in L^1(\ov M)$ then 
\begin{equation}\label{eq: L^1 bound Delta|DU|^2}
    z^{1-2s} \Delta |\ov \nabla U|^2 \in L^1(\ov M) .  
\end{equation}
\end{lemma}

\begin{proof} Observe that $V:=(-\Delta)^{1/2} U$ is the CS extension of $(-\Delta)^{1/2} u \in H^s(M)$. Then, regarding \eqref{eq: L^1 bound d_z Du} we have
\begin{align*}
    \int_{\ov M}  z^{1-2s}  |\nabla \partial_z U|^2  d\mu dz &= \int_0^\infty z^{1-2s} \|\nabla \partial_z U(\cdot,z)\|_{L^2(M)}^2 dz , 
\end{align*}
but for fixed $z>0$
\begin{equation*}
   \|\nabla \partial_z U(\cdot,z)\|_{L^2(M)}^2 =  \|(-\Delta)^{1/2} \partial_z U (\cdot,z) \|_{L^2(M)}^2   = \| \partial_z (-\Delta)^{1/2} U (\cdot,z) \|_{L^2(M)}^2= \| \partial_z V (\cdot,z) \|_{L^2(M)}^2 , 
\end{equation*}
hence 
\begin{equation*}
    \int_{\ov M}  z^{1-2s}  |\nabla \partial_z U|^2  d\mu dz = \int_{\ov M}  z^{1-2s}  |\partial_z V|^2  d\mu dz \lesssim  [(-\Delta)^{1/2} u]_{H^s(M)}^2 =  \|(-\Delta)^{\frac{1+s}{2}} u\|_{L^2(M)}^2 <\infty .
\end{equation*}

Now we prove \eqref{eq: L^1 bound Delta|DU|^2}. By Bochner's formula 
   \begin{equation*}
       \frac{1}{2}z^{1-2s}\Delta |\ov \nabla U|^2 =  z^{1-2s} \ov \nabla U \cdot \ov \nabla \Delta U + z^{1-2s}|\nabla \partial_z U|^2 +  z^{1-2s}\big(|\nabla^2 U|^2+\Ric(\nabla U, \nabla U) \big)
   \end{equation*}
The second term on the r.h.s. lies in $L^1(\ov M)$ by \eqref{eq: L^1 bound d_z Du}, and by hypothesis the last term (with the horizontal Hessian and the Ricci curvature) lies in $L^1(\ov M)$ too. 

We are left to show that $z^{1-2s} \ov \nabla U \cdot \ov \nabla \Delta U \in L^1(\ov M)$, but this easily follows by Cauchy-Schwarz since $\Delta U$ is the extension of $\Delta u$ and $u, \Delta u \in H^s(M)$. 
    
\end{proof}

\begin{lemma}\label{lem: density in Hs}
Let $s \in (0,1)$ and \(M\) be a complete Riemannian manifold without boundary. Then \(C_c^\infty(M)\) is dense in \(H^s(M)\). 
\end{lemma}

\begin{proof}
Since \(M\) is complete, by a classical result (e.g. \cite{Aubin}) there holds that \(C_c^\infty(M)\) is dense in \(W^{1,2}(M)\). Moreover, on a complete Riemannian manifold, the Laplacian with initial domain \(C_c^\infty(M)\) is essentially self-adjoint; see \cite{StrichartzJFA}.  Hence the operator \(I-\Delta\) is a positive self-adjoint operator on \(L^2(M)\), and its quadratic form domain is \(W^{1,2}(M)\). Therefore $H^1(M)=W^{1,2}(M)$ with equivalent norms, and it follows that \(C_c^\infty(M)\) is dense in \(H^1(M)\).

Now let $T:=(I-\Delta)^{1/2}$. Then \(T\) is positive self-adjoint on \(L^2(M)\), \(D(T)=H^1(M)\), and $H^s(M)=D(T^s)$ since $(1+\lambda^s)$ and $(1+\lambda)^s$ are comparable for $s\in (0,1)$. By the interpolation theorem for positive self-adjoint operators
\[
H^s(M) = D(T^s)=[L^2(M),D(T)]_s=[L^2(M),H^1(M)]_s,
\]
with equivalence of norms. Since $C_c^\infty(M)$ is dense in both $L^2(M)$ and $H^1(M)$, the standard density theorem for interpolation spaces (see \cite[Section 3.11]{IntSpaces}) gives that \(C_c^\infty(M)\) is dense $H^s(M)$.
\end{proof}

\vspace{10pt}

\noindent \textbf{Acknowledgements.} The authors would like to thank Emanuele Caputo, Mattia Freguglia, and Nicola Picenni for valuable comments on a preliminary draft of this work.

M.C. acknowledges the University of Warwick for the kind hospitality during the realization of part of this work. 

L.G. is supported by UK Research and Innovation (UKRI) under the Horizon Europe funding guarantee (grant agreement No.\ EP/Z000297/1).

	 	 	\bibliography{references}
	 	 \bibliographystyle{alpha}

\end{document}